\documentclass[12pt,a4paper,reqno]{amsart}
\usepackage[hmargin=2.5cm,bmargin=2.5cm,tmargin=3cm]{geometry}
\usepackage{enumerate}
\usepackage{amsmath,amsthm,amssymb,amsfonts,latexsym}
\usepackage[gen]{eurosym}
\usepackage[german,english]{babel}

\usepackage{latexsym}

\usepackage{tikz}					

     \newcommand{\NN}{\mathbb{N}}

     \newcommand{\RR}{\mathbb{R}}

\newtheorem{theorem}{Theorem}
\newtheorem{lemma}[theorem]{Lemma}

\newtheorem{proposition}[theorem]{Proposition}

\newtheorem{remark}[theorem]{Remark}

\newcommand{\be}{\begin{equation}}
\newcommand{\ee}{\end{equation}}
\newcommand{\bea}{\begin{eqnarray*}}
\newcommand{\eea}{\end{eqnarray*}}
\newcommand{\beq}{\begin{eqnarray}}
\newcommand{\eeq}{\end{eqnarray}}

\usepackage{todonotes}

\title[On the spectral gap of higher-dimensional Schrödinger operators]{On the spectral gap of higher-dimensional Schrödinger operators on large domains} 

\subjclass[2010]{}

\keywords{}

\author[J.~Kerner]{Joachim Kerner}
\author[M.~T\"aufer]{Matthias T\"aufer}

\address{Joachim Kerner, Lehrgebiet Analysis, Fakult\"at Mathematik und Informatik, Fern\-Universit\"at in Hagen, D-58084 Hagen, Germany}
\email{joachim.kerner@fernuni-hagen.de}

\address{Matthias T\"aufer, Lehrgebiet Analysis, Fakult\"at Mathematik und Informatik, Fern\-Universit\"at in Hagen, D-58084 Hagen, Germany}
\email{matthias.taeufer@fernuni-hagen.de}

\date{\today}

\thanks{
}

\begin{document}
	
\begin{abstract} We study the asymptotic behaviour of the spectral gap of Schrödinger operators in two and higher dimensions and in a limit where the volume of the domain tends to infinity. Depending on properties of the underlying potential, we will find different asymptotic behaviours of the gap. In some cases the gap behaves as the gap of the free Dirichlet Laplacian and in some cases it does not.  
\end{abstract}

\maketitle

\section{Introduction}
We investigate the spectral gap, that is the difference between the lowest two eigenvalues of two- and higher-dimensional Schrödinger operators for certain (non-negative and bounded) potentials in a limit where the volume of the underlying domain tends to infinity. 
Investigations of the spectral gap have a long history in spectral theory as can be seen from \cite{KirschGapII,KirschGap,AB,Abramovich,Abramovich} and references therein. An interesting result in this context is the recent
 proof of the fundamental gap conjecture by Andrews and Clutterbuck~\cite{AndrewsClutterbuck}. Note that the fundamental gap conjecture states that the spectral gap of the Laplacian with a weakly convex potential and defined on a bounded convex domain in $\mathbb{R}^d$ has a gap which is larger or equal to the gap of the free one-dimensional Dirichlet Laplacian on an interval whose length equals the diameter of the convex domain.

Our paper is motivated by a recent observation on the behaviour of the spectral gap of one-dimensional Schrödinger operators on intervals in the limit where the length of the interval tends to infinity. More explicitly, in \cite{KernerTaufer} it was shown that the spectral gap converges to zero at a strictly faster rate than the gap of the free Laplacian if the potential decays fast enough at infinity. This includes, in particular, compactly supported potentials. This somewhat surprising effect is a consequence of the effective degeneracy of the ground state in the infinite-volume limit. Loosely speaking, in the limit of infinite volume, even the smallest potential splits the interval into two congruent subintervals with an extra Dirichlet condition in the middle. Hence, the corresponding limiting operator has an effectively degenerate ground state which forces the spectral gap to close faster than the gap of the free Laplacian. Note that an analogous phenomenon has since been observed in~\cite{KernerPavlo} for Schrödinger operators on discrete graphs. Also, note that there are still open questions related to this phenomenon:  for example, one would like to determine the exact decay rate of the spectral gap. A conjecture in this direction was put forward in~\cite{KernerTaufer} and first results were subsequently obtained in~\cite{KernerPolyn}.

It is our aim to investigate if the effect found in one dimension persists in two and higher dimensions. More explicitly, we study Schrödinger operators on squares and hypercubes with side lengths going to infinity and with non-negative and bounded potentials that fulfil certain decay conditions. 
While it is relatively straightforward to see that, in dimensions larger than three, fast decaying (e.g.~integrable) potentials cannot change the asymtotics of the spectral gap, the two-dimensional situation turns out to be more subtle.
Even though it requires more technical effort, we prove that fast decaying potentials cannot change the asymptotic decay rate of the spectral gap.
We can push this result up to the borderline case of potentials that decay exactly like $\lvert x \rvert^{-2}$, assuming the potential is sufficiently small.
If the potential of a two-dimensional Schr\"odinger operator does decays in one direction but remains constant in the other, we prove -- again with a larger technical effort than in the one-dimensional situation -- that the ground state will be again effectively degenerate in the large volume limit. As a consequence, the faster decay rate of the spectral gap will reappear.

The rest of the paper is structured as follows: 
At the beginning of Section~\ref{sec:results} the setting and basic notation is introduced.
Subsection~\ref{subsec:two-dimensional} contains our results and proofs in two dimensions whereas Subsection~\ref{subsec:higher-dimensional} is about the higher-dimensional situation.
Some technical aspects regarding Bessel functions are deferred to the appendix.

 
\section{Results}

\label{sec:results}

On $\Lambda_L = \left(- \frac{L}{2}, \frac{L}{2} \right)^d \subset \RR^d$ with $L > 0$ and $d \geq 2$, we consider self-adjoint Schrödinger operators of the form
\begin{equation}\label{Hamiltonian}
h_L:=-\Delta+v
\end{equation}
where $-\Delta$ is the two-dimensional negative Laplace operator with Dirichlet boundary conditions on $\partial \Lambda_L$, and $v$ an external potential. 
We always assume that $v \in L^{\infty}(\mathbb{R}^d)$ and that $v$ is non-negative, i.e., $v \geq 0$.

The operator $h_L$ is defined on the domain $H^1_0(\Lambda_L)\cap H^2(\Lambda_L)$ in the Hilbert space $L^2(\Lambda_L)$ where $H^k(\Omega)$ is the Sobolev space of all square integrable functions on $\Omega$ with square integrable weak derivatives up to order $k$, and $H^1_0(\Omega)$ denotes the space of all functions in $H^1(\Omega)$ with zero boundary trace.
The operators $h_L$ are nonnegative, self-adjoint operators with purely discrete spectrum.
We denote their eigenvalues by $\lambda_0(L) < \lambda_1(L) \leq \dots $. 
The operator with zero potential, which is simply the Dirichlet Laplacian on $\Lambda_L$, shall be denoted by $h^{0}_L$ and its eigenvalues by $\mu_0(L) \leq \mu_1(L) \leq \dots $. 
The main object of interest in this paper is the \emph{spectral gap} which is defined by
\[
	\Gamma_v(L)
	:=
	\lambda_{1}(L) - \lambda_{0}(L).
\]
For a measurable set $A \subset \mathbb{R}^d$, we shall write $\mathbf{1}_A$ for the associated characteristic function. We write $B_r=B_r(0)$ for the $d$-dimensional ball of radius $r > 0$. For details on Bessel functions, we refer to the appendix. Here, we only remark that $J_\alpha$ is the Bessel function of the first kind with parameter $\alpha \geq  0$. Its zeroes form the increasing sequence $0 < j_{\alpha,1} < j_{\alpha,2} < \dots$.
\subsection{The two-dimensional case}
	\label{subsec:two-dimensional}

For $d=2$, we have
\begin{equation*}
\mu_0(L) = \frac{2 \pi^2}{L^2},
\quad
\mu_1(L) = \frac{(1^2 + 2^2) \pi^2}{L^2} = \frac{5 \pi^2}{L^2}
\end{equation*}
with associated normalized eigenfunctions
\begin{align*}
	\phi_{0}(x,y)
	&=
	\frac{2}{L} 
	\cos 
	\left(
		\frac{\pi x}{L}
	\right)
	\cdot
	\cos
	\left(
		\frac{\pi y}{L}
	\right),\\
	\phi_{1}(x,y)
	&=
	\frac{2}{L} 
	\sin 
	\left(
		\frac{2 \pi x}{L}
	\right)
	\cdot
	\cos
	\left(
		\frac{\pi y}{L}
	\right).
\end{align*}
Note that the eigenvalue $\mu_1(L)$ has multiplicity two: 
An orthogonal eigenfunction to $\phi_1$ with the same eigenvalue can be obtained by interchanging the roles of $x$ and $y$.

Let us introduce an auxiliary operator which is unitarily equivalent to $h_L$ and which is defined on the \textit{fixed} domain $\Lambda_1$. The $L$-dependence then manifests itself in the potential and in an additional factor. 
\begin{proposition}[Unitary transformation]\label{Unitary} In every dimension $d \geq 1$, the operator $h_L$ is unitarily equivalent to the operator  
\[
	L^{-2}g_L
	\quad
	\text{where}
	\quad
	g_L=-\Delta+L^2v(Lx)
\]
is defined in $L^2(\Lambda_{1})$ with Dirichlet boundary conditions.
\end{proposition}
\begin{proof} 
	This follows immediately from a direct calculation, cf.~\cite{KernerTaufer} for the case $d = 1$.
\end{proof}
In a first step we study potentials that decay sufficiently fast at infinity. 
Indeed, we will impose the same decay condition as in one dimension (see~\cite[Theorem~2.5]{KernerTaufer}) where it was shown that the spectral gap of the one-dimensional analogue of $h_L$ decays faster than the gap of the free Laplacian. 
It turns out, however, that this effect is absent in two dimensions.
\begin{theorem}[Gap for fast decaying potentials]\label{MainTheorem}
	Let $0 \leq v \in L^\infty(\RR^2)$ and assume
	\begin{equation*}
	|v(x)| \leq \frac{C}{|x|^{2+\mu}}
	\end{equation*}
	for some $C, \mu > 0$. 
	Then, there are $\alpha, \beta > 0$ such that for sufficiently large $L$
	\[
	\frac{\alpha}{L^2} \leq \Gamma_v(L) \leq \frac{\beta}{L^2}. 
	\]
\end{theorem}

\begin{proof} By Proposition~\ref{Unitary} it suffices to find $\tilde{\alpha},\tilde{\beta} > 0$ such that 
	\begin{equation}\label{XXX}
	\tilde{\alpha} < \tilde{\Gamma}_v(L) < \tilde{\beta}
	\end{equation}
	for sufficiently large $L$ where $\tilde{\Gamma}_v(L):=\tilde{\lambda}_1(L) -\tilde{\lambda}_0(L) $ denotes the spectral gap of $g_L$. 
	Since $v \geq 0$, we clearly have 
	\begin{equation*}
	5\pi^2 \leq \tilde{\lambda}_1(L)\ .
	\end{equation*}
	Thus, in order to establish the lower bound in \eqref{XXX}, it suffices to show $\tilde{\lambda}_0(L) < 5\pi^2$ for large enough $L > 0$. 
	Let $\delta > 0$ be such that $B_{\delta} \subset \Lambda_1$. 
	Then, for any $\varepsilon > 0$ we have, for sufficiently large $L$,
	\begin{equation}\label{eqproof}
	\|L^2 v (L \cdot) \mathbf{1}_{\Lambda_1 \setminus B_{\delta}}\|_{\infty} 
	\leq
	\frac{C L^2}{L^{2 + \mu} \delta^{2 + \mu}}
	< \varepsilon\ .
	\end{equation}
	By standard operator bracketing arguments,~\eqref{eqproof} implies
	\begin{equation}\label{GGG}
	\tilde{\lambda}_0(L) \leq \tilde{\mu}_{0,\delta}+\varepsilon
	\end{equation} 
	where $\tilde{\mu}_{0,\delta}$ is the ground-state eigenvalue of the Dirichlet Laplacian in $L^2 \left( B_{\frac{1}{2}} \setminus B_{\delta} \right)$. 
	Now, a result by Ozawa \cite{Ozawa} shows that $\tilde{\mu}_{0,\delta}$ converges to the ground-state eigenvalue of the Dirichlet Laplacian on $B_{\frac{1}{2}}$.
	The latter is given by $4j^2_0$ where $j_0$ is the first positive root of the Bessel function $J_0$. 
	Since $\varepsilon$ in \eqref{GGG} was arbitrary and since $4j^2_0 < 5\pi^2$, the lower bound follows. 
	
	In the same way, one has
	\begin{equation*}
	|\tilde{\Gamma}_v(L)| \leq |\tilde{\lambda}_1(L)|  \leq |\nu_1(\delta)|+\varepsilon 
	\end{equation*} 
	where $\nu_1(\delta)$ is the second eigenvalue of the Dirichlet Laplacian in $L^2 \left( B_{\frac{1}{2}} \setminus B_{\delta} \right)$. This proves the upper bound.
\end{proof}
Theorem~\ref{MainTheorem} is about potentials that decay faster than quadratically. 
In particular, it shows that the effect from one dimension, meaning the faster convergence of the spectral gap due to arbitrarily small potentials, cannot be reproduced in two dimensions. 
\begin{remark}
	There is also a heuristic explanation of Theorem~\ref{MainTheorem} in terms of capacities, referring to the fact that single points in two and higher dimensional Euclidean space have harmonic (or Newtonian) capacity zero; see~\cite{TaylorR75} for a reference.
	This means that the infimum of the Dirichlet form $\int \lvert \nabla u \rvert^2 \ \mathrm{d}x$ over all $u \in C^{\infty}_0(\mathbb{R}^d)$ which satisfy $u(0) = 1$ and decay at $\infty$ is zero if $d \geq 2$.
	Since, using the scaling of Proposition~\ref{Unitary}, compactly supported potentials are essentially reduced to a single point, it is unsurprising that they do not have the power to warp the asymptotics of the spectral gap whereas other objects which macroscopically look like objects with positive capacity, do, as we will see in Theorem~\ref{thm:one-sided} below.
\end{remark}
We now turn to the borderline case of Theorem~\ref{MainTheorem}, that is to potentials that can decay quadratically.
Since a potential $\frac{1}{|x|^2}$ is left invariant under the unitary transformation of Proposition~\ref{Unitary}, the proof will have to be amended accordingly.
We start with an upper bound on the gap.
\begin{theorem}
	\label{thm:simple_lower_bound}
	Let $0 \leq v \in (L^1 \cap L^{\infty})(\mathbb{R}^2)$ be such that 
	\begin{equation*}
	|v(x)| \leq \frac{C}{|x|^2}
	\end{equation*}
	for some $C > 0$ and almost all $x \in \mathbb{R}^2$. Then, there is $\beta > 0$ such that, for sufficiently large $L$, one has
	\[
	 \Gamma_v(L) \leq \frac{\beta}{L^2}\ .
	\]
\end{theorem}
\begin{proof} We use the unitary transformation given in Proposition~\ref{Unitary} and note that $\tilde{\Gamma}_v(L)$ is bounded by the second eigenvalue of the Schrödinger operator $$h_{\delta}:=-\Delta+\frac{C}{|x|^2}\ ,$$
	defined in $L^2 \left( B_{\frac{1}{2}} \setminus B_{\delta} \right)$ with Dirichlet boundary conditions and with $\delta > 0$ sufficiently small. 
\end{proof}
Next, we complement Theorem~\ref{thm:simple_lower_bound} by a \emph{lower bound} on the spectral gap for potentials that decay \emph{exactly} quadratically at infinity. Recall that $J_\alpha$ is the Bessel function of the first kind with parameter $\alpha \geq  0$ and zeroes $0 < j_{\alpha,1} < j_{\alpha,2} < \dots$. 


\begin{theorem}[Regularized quadratic potentials]
	\label{thm:quadratic_potential}
	Let $c > 0$ be such that 
	\begin{equation}
	\label{eq:condition_Bessel}
	j_{c,1} < \frac{ \min \{ j_{\sqrt{1+c},1}, j_{\sqrt{c},2} \}}{\sqrt{2}}\ ,
	\end{equation} 
	and let
	\[
	v(x) = \frac{c \mathbf{1}_{\lvert x \rvert \geq 1}}{\lvert x \rvert^2}\ .
	\]
	Then, there are constants $\alpha, \beta > 0$ such that, for sufficiently large $L$, one has
	\[
	\frac{\alpha}{L^2}
	\leq
	\Gamma_v(L)
	\leq
	\frac{\beta}{L^2}\ .
	\]
\end{theorem}
\begin{remark}
	Condition~\eqref{eq:condition_Bessel} holds for sufficiently small $c$; for instance, for all $c < 0.0468\dots$.
\end{remark}

\begin{proof}
	The upper bound was established in Theorem~\ref{thm:simple_lower_bound}. As for the lower bound, recall that $\tilde \lambda_0(L)$ and $\tilde \lambda_1(L)$ denote the eigenvalues of the operator
	\[
	g_L = - \Delta + \frac{c \mathbf{1}_{\lvert x \rvert \geq \frac{1}{L}}}{\lvert x \rvert^2}
	\quad
	\text{on $L^2(\Lambda_1)$ with Dirichlet boundary conditions}\ .
	\]
	By Proposition~\ref{Unitary}, it then suffices to prove $\tilde \Gamma_v(L) = \tilde \lambda_1(L) - \tilde \lambda_0(L) \geq \alpha > 0$ for some $ \alpha > 0$ and for sufficiently large $L$.
	
	 For this purpose we use Lemma~\ref{lem:Bessel} from the appendix while considering the operator $\tilde g_{\rho,L,c}:=- \Delta + \frac{c \mathbf{1}_{\lvert x \rvert \geq \frac{1}{L}}}{\lvert x \rvert^2}$ on balls $B_{\rho}$ with radii $\rho=1/2$, $\rho=1/\sqrt{2}$ and subject to Dirichlet boundary conditions. 
	 The asymptotic expressions for the eigenvalues of $\tilde g_{\rho,L,c}$ in Lemma~\ref{lem:Bessel} then yields
	\begin{align*}
	\lim_{L \to \infty}
	\tilde \lambda_1(L) - \tilde \lambda_0(L)
	&\geq
	\lim_{L \to \infty}
	\nu_1 ( L, 1/\sqrt{2}, c )
	-
	\nu_0 ( L, 1/2, c )\\
	&=
	2 \min \{ j_{\sqrt{1+c},1}, j_{\sqrt{c},2} \}^2
	-
	4 j_{\sqrt{c},1}^2
	>
	0
	\end{align*}
	as $L \to \infty$,
	where $\nu_k ( L, \rho, c )$, $k=0,1$, are the lowest two eigenvalues of $\tilde g_{\rho,L,c}$, and the last inequality follows from~\eqref{eq:condition_Bessel}.
\end{proof}

\begin{remark}
	Although we believe that Theorem~\ref{thm:quadratic_potential} actually holds true for all values $c > 0$, let us comment in more detail on the origin of condition~\eqref{eq:condition_Bessel}:
	Both, $j_{\sqrt{c},1}^2$ and $\min \{j_{\sqrt{1 + c},1}, j_{\sqrt{c},2} \}^2$, correspond to the two smallest $\lambda$ solving the formal eigenvalue equation $- \Delta u + \frac{c}{\lvert x \rvert^2} u = \lambda u$ on $B_1$ with Dirichlet boundary conditions. 
	General properties of Bessel functions, as listed in Appendix~\ref{sec:Bessel}, trivially imply $\min \{j_{\sqrt{1 + c},1}, j_{\sqrt{c},2} \} - j_{\sqrt{c},1} > 0$.
One could, however, improve on the range of admissible values of $c$ if the divisor $\sqrt{2}^{-1}$ in~\eqref{eq:condition_Bessel} was closer to one.

	The factor $\sqrt{2}^{-1}$ stems from the geometric argument in the proof of Theorem~\ref{thm:quadratic_potential} where eigenvalues on $\Lambda_1$ are estimated by eigenvalues on balls with inradius and outer radius of $\Lambda_1$ and is exactly the ratio between the outer and inner radius of $\Lambda_1$.
	
	Consequently, if $\Lambda_L$ was replaced by a scaled version $K_L := \{ x \in \RR^d\colon x/L \in K \}$ of another set $K \subset \RR^d$, the factor $\sqrt{2}^{-1}$ would be replaced by
	\[
	\mathcal{R}(K)
	:=
	\frac{\inf \{ r > 0: B_r \subset K \}}{\sup \{ r > 0 : K \subset B_r \}}.
	\]
	In Table~\ref{table:c} we sketch the consequences on the range of admissible $c$ for some domains $K \subset \RR^2$.
\begin{table}[h]
\begin{tabular}{ c|c|c } 
 Shape $K$ & $\mathcal{R}(K)$ & maximal $c$ for quadratic lower bound on gap \\ 
  \hline
  \hline
 square & $\cos(\pi/4)$ & $0.0468\dots$ \\ 
 \hline
 regular hexagon & $\cos(\pi/6)$ & $0.8719\dots$ \\
 \hline
 regular octagon & $\cos(\pi/8)$ & $2.4379\dots$ \\
 \hline
 ball & $1$ & $\infty$ 
\end{tabular}
\caption{Range of parameters $c$ for which we can prove a quadratic lower bound on the spectral gap of the operator $- \Delta + \frac{c \mathbf{1}_{\lvert x \rvert \geq 1/L}}{\lvert x \rvert^2}$ on different domains.}
\label{table:c}
\end{table}

\end{remark}

Our results so far have been negative: in contrast to the one-dimensional setting, in two dimensions the asymptotic decay rate of the gap seems robust. Fast decaying potentials are unable to warp the decay rate of the spectral gap.
But one can still reproduce the effective degeneracy of the ground state by using potentials on a strip in two dimensions. For such potentials, we prove an analogue of the one-dimensional case described in \cite[Theorem~2.5]{KernerTaufer} and show that the gap converges to zero strictly faster than the spectral gap of the free Laplacian. 
\begin{theorem}[One-sided decay] 
\label{thm:one-sided}
Assume that $0 \leq v \in L^{\infty}(\mathbb{R}^2)$ is a potential such that $v(x,y) \geq \gamma > 0$ for $-\delta < x < +\delta$, $\delta > 0$, and which is zero elsewhere. Then, one has
	\begin{equation*}
	\lim_{L \rightarrow \infty}L^2\Gamma_v(L)=0\ . 
	\end{equation*}
\end{theorem}
\begin{proof} 
By Proposition~\ref{Unitary}, it suffices to show $\lim_{L \rightarrow \infty}\tilde{\Gamma}_v(L)=0$ for the gap $\tilde{\Gamma}_v(L)$ of the operator $g_L$. 
	
	Let us first pick two test functions: 
	\begin{align*}
	\phi_1^{(L, \epsilon)}(x,y)
	&=
	\begin{cases}
	\sin \left( \frac{\pi}{\frac{1}{2} - \epsilon} (x -\epsilon) \right) \cos (\pi y)
	&
	\text{if}\ \epsilon \leq x \leq \frac{1}{2},\\
	0
	&
	\text{if}\ - \frac{1}{2} \leq x < \epsilon,
	\end{cases}
	\\
	\text{and}
	\quad
	\phi_2^{(L, \epsilon)}
	(x,y)
	&=
	\phi_1^{(L, \epsilon)} 
	(-x,y)
	\end{align*}
	which are both in $H_0^1(\Lambda_1)$ and satisfy Dirichlet boundary conditions on $\partial \Lambda_1$. 
	Plugging $\phi_{k}^{(L, \epsilon)}$, $k=1,2$, into the Rayleigh quotient for the quadratic form corresponding to $g_L$, we find, for every $\delta > 0$, parameters $\epsilon_0, L_0 > 0$ such that for $\epsilon \geq \epsilon_0$ and $L \geq L_0$ one has $\tilde \lambda_1(L) \leq 5 \pi^2 + \delta$. 
	Since $\delta > 0$ is arbitrary, we conclude
		\begin{equation*}
	\limsup_{L \rightarrow \infty}\tilde{\lambda}_1(L) \leq 5 \pi^2
	\end{equation*}
	where $\tilde{\lambda}_1(L)$ is the second eigenvalue of $g_L$.
	It remains to prove
	\begin{equation}\label{XXXx}
	\liminf_{L \rightarrow \infty}\tilde{\lambda}_0(L) \geq 5 \pi^2\ ,
	\end{equation}
	where $\tilde{\lambda}_0(L)$ is the ground state eigenvalue of $g_L$: 
	Let $L_k \to \infty$ and take a sequence of corresponding ground states $\varphi_k \in H^1_0(\Lambda_1)$ of $ g_{L_k}$.
	Obviously, since $\tilde \lambda_0(L)$ is bounded as $L \to \infty$, the sequence $(\varphi_k)$ is bounded in $H^1 (\Lambda_1)$. 
	We extract a subsequence that converges weakly in $H^1_0(\Lambda_1)$ and strongly in $L^2(\Lambda_1)$ to a function $\varphi_{\infty}$. 
	For every $n \in \NN$, on the rectangles $\mathfrak{L}_n:=\left(-\frac{1}{2},-\frac{1}{n} \right) \times (-\frac{1}{2},\frac{1}{2})$, and $\mathfrak{R}_n:=\left(\frac{1}{n},\frac{1}{2} \right) \times (\frac{1}{2},\frac{1}{2})$ the elliptic regularity estimate 
	$\|u\|_{H^2(\Omega)} \leq C\left(\|\Delta u\|_{L^2(\Omega)}+\| u\|_{L^2(\Omega)}\right)$, see, e.g., \cite{Gri85}, implies
	\begin{equation*}
	\|\varphi_k\|_{H^2(\mathfrak{R}_n)} \leq C
	\quad
	\text{and}
	\quad
	\|\varphi_k\|_{H^2(\mathfrak{L}_n)} \leq C 
	\end{equation*}
	for sufficiently large $k$, and where the constant $C > 0$ might depend on $n \in \mathbb{N}$. 
	
	Let us from now on focus on $\mathfrak{R}_n$ since the subsequent arguments for $\mathfrak{L}_n$ will be completely analogous.
	The uniform $H^2$-boundedness of the $\varphi_k$ on $\mathfrak{R}_n$ for every fixed $n \in \NN$ implies -- again after passing to a subsequence -- that
	\[
	\varphi_\infty
	\in
	H^1(\mathfrak{R}_n)
	\quad
	\text{for all $n \in \NN$}
	\]
	with an $n$-independent bound on $\lVert \phi_\infty \rVert_{H^1(\mathfrak{R}_n)}$.
	Therefore,
	\[
	\varphi_{\infty} 
	\in 
	H^1 \left(
			\left(0, \frac{1}{2} \right) 
			\times
			\left(-\frac{1}{2}, \frac{1}{2} \right) 
		\right) 
		=:
		H^1(\mathfrak{R})\ .
	\]
	Next, we define 
	\[
	\left[ 0, \frac{1}{2} \right] \ni x 
	\mapsto 
	f_k(x)
	:=
	\int_{-1/2}^{1/2}
	\lvert \varphi_k(x,y) \rvert^2
	\mathrm{d} y
	\quad
	\text{for $k \in \NN \cup \{ \infty \}$}.
	\]
	By the trace theorem and uniform $H_1(\mathfrak{R})$-boundedness of the $\varphi_k$,
	\[
	\lvert
	f_k (x)
	\rvert
	\leq
	C_1
	\lVert \varphi_k \rVert_{H^1(\mathfrak{R})}
	\leq
	C_2
	\]
	with constants $C_1,C_2$ that are independent of $k$ and $x$.
	We estimate, for $x \in (0,1/2)$ and sufficiently small $h$,
	\begin{align*}
	f_k(x + h) &- f_k(x) 
	\\
	&=
	\int_{-1/2}^{1/2}
	\left( \varphi_k(x+h,y) - \varphi_k(x,y) \right)
	\cdot
	\left( \varphi_k(x+h,y) + \varphi_k(x,y) \right)
	\mathrm{d} y
	\\
	&=
	\int_{-1/2}^{1/2}
	\left(
		\int_x^{x+h} \partial_x \varphi_k(t,y) 
		\mathrm{d} t
	\right)
	\cdot
	\left( \varphi_k(x+h,y) + \varphi_k(x,y) \right)
	\mathrm{d} y
	\\
	&\leq
	\left(
	\int_{-1/2}^{1/2}
		\left[
			\int_x^{x+h}
			\partial_x \varphi_k(t,y)
			\mathrm{d} t
		\right]^2
	\mathrm{d} y
	\right)^{1/2}
	\left( 
		2 
		\int_{-1/2}^{1/2}
		\lvert \varphi_k(x+h,y) \rvert^2
		+
		\lvert \varphi_k(x,y) \rvert^2
		\mathrm{d} y
	\right)^{1/2}
	\\
	&\leq
	C
	\lVert \varphi_k \rVert_{H^1(\mathfrak{R})}^{1/2}
	\cdot
	\left(
	h
	\int_{-1/2}^{1/2}
			\int_x^{x+h}
				\lvert
				\partial_x \varphi_k(t,y)
				\rvert^2
			\mathrm{d} t
	\mathrm{d} y
	\right)^{1/2}
	\leq
	C_3
	\sqrt{h}
	\lVert \varphi_k \rVert_{H^1(\mathfrak{R})}
	\end{align*}
	where we used the Cauchy-Schwarz inequality and Jensen's inequality. Furthermore, $C_3 > 0$ is independent of $k,x,h$.
	We conclude that the family $(f_k)_{k \in \NN}$ is 
equicontinuous and pointwise bounded on a compact interval.
	By the Arzel\'a Ascoli theorem, and since $f_k \to f_\infty$ pointwise almost everywhere, we can therefore -- again passing to a subsequence -- assume that $f_k \to f_\infty$ uniformly in $x$.
	
	We now claim $f_\infty(0) = 0$:
	Indeed, if $f_\infty(0) > 0$, we would have $f_\infty(x)> \kappa$ on an interval $[0, \epsilon)$, for some $\epsilon, \kappa > 0$.
	But then
	\begin{align*}
	\left\langle \varphi_k, g_{L_k} \varphi_k \right\rangle
	&\geq
	\int_{\mathfrak{R}} L_k^2 v(L_k x,L_k y) \lvert \varphi_k(x,y) \rvert^2 \mathrm{d} x \mathrm{d} y
	\geq
	\gamma L_k^2
	\int_0^{\delta/L_k}
	f_k(x) 
	\mathrm{d} x \\
	&\geq
	\frac{\delta \kappa \gamma}{2} L_k
	\to \infty
	\end{align*}
	for sufficiently large $k$ by uniform convergence.
	However, this would contradict the boundedness of the ground state eigenvalue.
	As a consequence, $f_\infty(0) = 0$ which means that $\phi_\infty \in H^1(\mathfrak{L} \oplus \mathfrak{R})$ satisfies Dirichlet boundary conditions on the entire boundary $\partial \mathfrak{R}$ and -- by a completely analogous argument -- on $\partial \mathfrak{L}$.
	But since the ground state energy of the Dirichlet Laplacian on these domains is $5 \pi^2$, this shows~\eqref{XXXx} and the proof is complete.
\end{proof}


\subsection{The higher-dimensional case}
	\label{subsec:higher-dimensional}

We now consider the operator $h_L$, defined in~\eqref{Hamiltonian}, on $\Lambda_L \subset \RR^d$ for $d \geq 3$. 
If the potential $v$ is identically zero, the least two eigenvalues are 
\begin{equation*}
\mu_0(L)=\frac{d\pi^2}{L^2} \qquad \text{and} \qquad \mu_1(L)=\frac{(d+3)\pi^2}{L^2}
\end{equation*}
where $\mu_1$ is $d$-fold degenerate.
Also, the ground state eigenfunction of $h^0_L$ is
\begin{equation*}
\phi_{0}(x_1,...,x_d)= \left(\frac{2}{L}\right)^{d/2}
\prod_{j = 1}^d
\cos\left(\frac{\pi x_j}{L}\right)
.
\end{equation*}
It is easier than in the two-dimensional setting to prove that fast decaying potentials cannot affect the decay rate of the spectral gap.
\begin{theorem}
	Let $d\geq 3$ and $h_L$ be given with an arbitrary non-negative potential $v \in (L^1 \cap L^{\infty})(\mathbb{R}^d)$. Then, 
 	\[
	\frac{\alpha}{L^2}
	\leq
	\Gamma_v(L) 
	\leq
	\frac{\beta}{L^2}
	\]
	for some constants $\alpha, \beta > 0$ and all $L > 0$ large enough.
\end{theorem}

\begin{proof} The upper bound follows immediately from the minmax-principle by looking at the two-dimensional subspace generated by the first two eigenstates of $h^0_L$ while taking into account that $\mu_0(L), \mu_1(L) \sim L^{-2}$.

	Regarding the lower bound we first notice that 
	\begin{equation*}
	\frac{(d+3)\pi^2}{L^2}=\mu_1(L) \leq \lambda_1(L) \ .
	\end{equation*}
On the other hand, 
	\[
	\lambda_{0}(L)
	\leq
	\left\langle
	\phi_{0},
	h_L
	\phi_{0}
	\right\rangle_{L^2(\Lambda_L)}
	\leq
	\frac{d \pi^2}{L^2}
	+
	\lVert \phi_{0} \rVert_\infty^2
	\cdot 
	\lVert v\rVert_{L^1(\mathbb{R}^d)}
	=
	\frac{d \pi^2}{L^2}
	+
	\frac{2^d \lVert v\rVert_{L^1(\mathbb{R}^3)}}{L^d}\ .
	\]
From this the statement follows immediately.
\end{proof}
\appendix

\section{Bessel functions and the operator $- \Delta + \lvert x \rvert^{-2}$}

\label{sec:Bessel}

Let us recall some facts on Bessel's differential equation with parameter $\alpha \geq 0$,
	\begin{equation}
	\label{eq:Bessel}
	r^2 R''(r) + r R(r) + [r^2 - \alpha^2] R(r) = 0,
	\quad
	r \in [0,\infty).
	\end{equation}
 	This one-dimensional ordinary differential equation naturally appears when the equation $- \Delta u = \lambda u$ is rewritten in polar coordinates.
There is a well-developed theory on solutions of~\eqref{eq:Bessel}. 
We collect some useful properties.

\begin{proposition}[{See for instance~\cite[Chapter~9]{AbramowitzS-64}}]
	\label{prop:Bessel}
	For every $\alpha \geq 0$ and on any interval $I \subset (0, \infty)$, Bessel's differential equation~\eqref{eq:Bessel} has a two-dimensional space of solutions, spanned by the Bessel function of the first kind $J_\alpha$ and of the second kind $Y_\alpha$.
	For every $\alpha \geq 0$ it holds that
	\begin{enumerate}[(i)]
		\item
		the function $J_\alpha$ is uniformly bounded and analytic in $(0,\infty)$\ ,
		\item
		the function $Y_\alpha$ is analytic in $(0, \infty)$ with a singularity at $x = 0$ of order $x^{- \alpha}$, if $\alpha > 0$, or a logarithmic singularity, if $\alpha = 0$\ ,
		\item
		one can expand around $x = 0$\footnote{We are not going to provide an asymptotic expansion of $Y_0$ and $Y_0'$ around $0$ since we won't use it below.}
	\begin{align*}
	J_\alpha(x)
	&\sim
	x^{\alpha} 
	\frac{2^{- \alpha}}{\Gamma(\alpha + 1)}	
	\quad
	&\text{if $\alpha > 0$, and}
	\quad
	J_0(x)
	\sim
	1\ ,
	\\
	J'_\alpha(x)
	&\sim
	x^{\alpha - 1} 
	\frac{2^{- \alpha}}{\Gamma(\alpha)}	
	\quad
	&\text{if $\alpha > 0$, and}
	\quad
	J'_0(x)
	\sim
	\frac{- x}{2}\ ,
	\\
	Y_\alpha(x)
	&\sim
	x^{- \alpha} 
	\frac{2^\alpha \Gamma(\alpha)}{\pi}
	\quad
	&\text{if $\alpha > 0$}\ ,
	\quad
	\\
	Y'_\alpha(x)
	&\sim
	x^{- \alpha - 1} 
	\frac{2^{\alpha} \Gamma(\alpha + 1)}{\pi} 	
	\quad
	&\text{if $\alpha > 0$}\ .
	\end{align*}
		\item
		both functions $J_\alpha$ and $Y_\alpha$ have an infinite sequence of isolated zeroes $(j_{\alpha,k})_{k = 1}^\infty$ and $(y_{\alpha,k})_{k = 1}^\infty$, tending to $\infty$, and
		\item
		$j_{\alpha,k}$ and $y_{\alpha,k}$ are continuous and monotonically increasing in $\alpha$.
	\end{enumerate}
\end{proposition}

For $c \geq 0$, let
\[
	\Sigma_c 
	:= 
	\bigcup_{n = 0}^\infty
	\bigcup_{k = 1}^\infty
	\left\{
		j_{\sqrt{n^2 + c},k}^2
	\right\}
	=:
	\left\{
	\nu_m(\infty, \rho, c)
	\colon
	m \in \{0, 1, 2, \dots \}
	\right\},
\]
where the $\nu_m(\infty, \rho, c)$ are the elements of the discrete set $\Sigma_c$, enumerated increasingly and counting multiplicities.
Proposition~\ref{prop:Bessel} implies 
\begin{equation}
	\label{eq:trivial_Bessel_estimate}
	\nu_0(\infty,1,c) 
	=
	j_{\sqrt{c},1}^2 
	< 
	\min 
	\{
	j_{\sqrt{c},2}^2,
	j_{\sqrt{c+1},1}^2
	\}
	=
	\nu_1(\infty,1,c)
	.
\end{equation}

The following lemma characterises $\Sigma_c$ as the limiting spectrum of the operator $- \Delta + \frac{c \mathbf{1}_{\lvert x \rvert \geq \frac{1}{L}}}{\lvert x \rvert^2}$ in $L^2(B_\rho)$ with Dirichlet boundary conditions as $L \to \infty$ .

\begin{lemma}
	\label{lem:Bessel}
	For any $\rho, L > 0$, every $c \geq 0$, and every $k \in \NN$, the $k$-th eigenvalue $\nu_k(L,\rho,c)$ of the operator
	\[
	\tilde g_{\rho,L,c}
	=
	- \Delta + \frac{c \mathbf{1}_{\lvert x \rvert \geq \frac{1}{L}}}{\lvert x \rvert^2}
	\quad
	\text{in $L^2(B_\rho)$ with Dirichlet boundary conditions} 
	\]
	satisfies
	\[
	\lim_{L \to \infty} \nu_k(L,\rho,c)
	=
	\frac{\nu_k(\infty,1,c)}{\rho^2}.
	\]
	\end{lemma}	

\begin{proof}[Proof of Lemma~\ref{lem:Bessel}]	
	The operator $\tilde g_{\rho,L,c}$ is radially symmetric which allows for a separation of variables. Writing $u(r,\vartheta) = R(r) \Theta(\vartheta)$, the eigenvalue equation $- \Delta u + \frac{c \mathbf{1}_{\lvert x \rvert \geq \frac{1}{L}}}{\lvert x \rvert^2} = \lambda u$ becomes
	\[ 
	- \Theta R'' - \frac{\Theta R'}{r} - \frac{\Theta'' R - c \mathbf{1}_{r \geq \frac{1}{L}} \Theta R}{r^2} = \lambda \Theta R\ ,
	\]
	which is equivalent to
	\[
	r^2 R'' + r R' + \left[ \lambda r^2 - \left(-\frac{\Theta''}{\Theta} + c \mathbf{1}_{\lvert x \rvert \geq \frac{1}{L}}\right) \right] R = 0 \ .
	\]
	Periodicity of the angular component $\Theta$ requires $\Theta(\vartheta) = \mathrm{e}^{i n \vartheta}$ for some $n \in \mathbb{Z}$, whence $-\frac{\Theta''}{\Theta} = n^2$ for $n = 0,1,2,\dots$. 
This finally leads to the ordinary differential equation, $n \in \mathbb{Z}$,
	\[
	r^2 R'' + r R' + \left[ \lambda r^2 - (n^2 + c \mathbf{1}_{\lvert x \rvert \geq \frac{1}{L}}) \right] R = 0
	\]
	with boundary conditions $R'(0) = 0$ and $R(\rho) = 0$. 
	After the transformation $r \mapsto \sqrt{\lambda} r$, we arrive at the system
	
	%
	\begin{equation}
	\label{eq:complicated_Bessel}
	\begin{cases}
	R'(0) = 0 &\\
	r^2 R'' + r R' + \left[ r^2 - n^2 \right] R = 0
	&
	\text{for $r \in [0, \sqrt{\lambda}/L]$}\ ,\\
	r \mapsto R(r) & \text{continuous and differentiable at $r = \sqrt{\lambda}/L$}\ ,\\
	r^2 R'' + r R' + \left[ r^2 - (n^2 + c) \right] R = 0
	&
	\text{for $r \in [\sqrt{\lambda}/L, \sqrt{\lambda} \rho]$}\ ,\\
	R(\sqrt{\lambda} \rho) = 0.
	\end{cases}
	\end{equation}
	The spectrum of $\tilde g_{\rho,L,c}$ can now be found by identifying, for every $n \in \NN$, the set of all $\lambda > 0$ such that~\eqref{eq:complicated_Bessel} has a solution.
	
	We remark that any eigenvalue of $\tilde g_{\rho,L,c}$ is bounded from below by $(j_{0,1}/\rho)^2 > 0$. Thus, for all $n \in \NN$ and $L > 0$, solutions of~\eqref{eq:complicated_Bessel} are such that 
	\begin{equation}
	\label{eq:no_small_zeroes}
	R(r) \neq 0
	\quad
	\text{for all}
	\quad
	r 
	\in	(0,j_{0,1}).
	\end{equation}
	Since Bessel functions of the second kind $Y_\alpha$ have a singularity at the origin, the Neumann condition at $r = 0$, forces any solution of~\eqref{eq:complicated_Bessel} to be of the form $R(r) = J_n(r)$ in the interval $[0, \sqrt{\lambda}/L]$ for some $n \in \NN$.
	On the interval $[\sqrt{\lambda}/L, \sqrt{\lambda} \rho]$, on the other hand, the solution of~\eqref{eq:complicated_Bessel} will be a linear combination $R(r) = a J_{\sqrt{n^2 + c}} + b Y_{\sqrt{n^2 + c}}$ with appropriate parameters $a,b$.
	Now, the matching conditions at $r = \sqrt{\lambda}/L$ imply
	\begin{equation}
	\label{eq:2x2_system}
	\begin{pmatrix}
	J_{\sqrt{n^2 + c}} ( \sqrt{\lambda}/L ) & Y_{\sqrt{n^2 + c}} ( \sqrt{\lambda}/L )\\
	J'_{\sqrt{n^2 + c}} ( \sqrt{\lambda}/L ) & Y'_{\sqrt{n^2 + c}} ( \sqrt{\lambda}/L )\\
	\end{pmatrix}
	\begin{pmatrix}
	a\\
	b\\
	\end{pmatrix}
	=
	\begin{pmatrix}
	J_{n} ( \sqrt{\lambda}/L )\\
	J'_{n} ( \sqrt{\lambda}/L )\\
	\end{pmatrix}.
	\end{equation}
	Recall that we are interested in the $k$-th eigenvalue of $\tilde g_{L,s,c}$ for large $L$. Clearly, this eigenvalue is bounded from above ($k$ different normalised test functions yield an upper bound by the min-max principle) whence we may assume that $\lambda$ is uniformly bounded from above by some $\lambda_{\max} > 0$ and this means that $\sqrt{\lambda}/L$ tends to zero.

	Now, solving the system~\eqref{eq:2x2_system} and dividing $a$ by $b$, we find
	\begin{equation}
	\label{eq:a/b}
	\frac{a}{b} 
	=
	\frac{
		Y'_{\sqrt{n^2 + c}} ( \sqrt{\lambda}/L ) J_n ( \sqrt{\lambda}/L ) 
		-
		Y_{\sqrt{n^2 + c}} ( \sqrt{\lambda}/L ) J'_n ( \sqrt{\lambda}/L )}
	{
		J_{\sqrt{n^2 + c}} ( \sqrt{\lambda}/L ) J'_n ( \sqrt{\lambda}/L )
		-
		J'_{\sqrt{n^2 + c}} ( \sqrt{\lambda}/L ) J_n ( \sqrt{\lambda}/L )
	}\ .
	\end{equation}
The asymptotics of the Bessel functions as in Proposition~\ref{prop:Bessel}~(iii) imply 
\[
	\Big\lvert \frac{a}{b} \Big\rvert 
	\sim 
	(\sqrt{\lambda}/L )^{- 2 (\sqrt{n^2 + c})} \to \infty
	\quad
	\text{as $L \to \infty$}.
\]
We now claim that the zeroes of $R(r) = a J_{\sqrt{n^2 + c}}(r) + b Y_{\sqrt{n^2 + c}} (r)$ in~\eqref{eq:complicated_Bessel} and on $[\sqrt{\lambda}/L, \sqrt{\lambda}\rho]$ will tend to the zeroes of $J_{\sqrt{n^2 + c}}$.
	This follows from~\eqref{eq:no_small_zeroes} and the a-priori upper bound $\lambda \leq \lambda_{\max}$ which implies that we can restrict the search for zeroes of $R$ to the \emph{compact} interval $[j_{0,1}, \rho \sqrt{\lambda_{\max}}] \subset (0,\infty)$ on which both, $J_{\sqrt{n^2 + c}}$ and $Y_{\sqrt{n^2 + c}}$, are analytic.
	Since $\lvert a/b \rvert \to \infty$, the statement follows.
\end{proof}

\vspace*{0.5cm}

{\small
	\bibliographystyle{amsalpha}
	\bibliography{Literature}}

\end{document}